\documentclass[12pt,reqno]{amsart}

\usepackage{times}
\usepackage[T1]{fontenc}
\usepackage{mathrsfs}
\usepackage{latexsym}
\usepackage[dvips]{graphics}
\usepackage{epsfig}
\usepackage{verbatim}

\usepackage{amsmath,amsfonts,amsthm,amssymb,amscd}
\input amssym.def
\input amssym.tex
\usepackage{color}
\usepackage{hyperref}
\usepackage{color}
\usepackage{breakurl}
\usepackage{comment}
\newcommand{\bburl}[1]{\textcolor{blue}{\url{#1}}}

\newcommand{\be}{\begin{equation}}
\newcommand{\ee}{\end{equation}}
\newcommand{\bea}{\begin{eqnarray}}
\newcommand{\eea}{\end{eqnarray}}

\newtheorem{thm}{Theorem}[section]
\newtheorem{conj}[thm]{Conjecture}
\newtheorem{cor}[thm]{Corollary}
\newtheorem{lem}[thm]{Lemma}
\newtheorem{prop}[thm]{Proposition}

\newtheorem{rek}[thm]{Remark}

\newcommand{\bb}[1]{\mathbb{#1}}

\newcommand{\sm}{\textup{Sym}^n f \otimes \chi}

\newcommand{\mxp}{\max_{t \in \bb{R}} |\phi(t)|}

\newcommand{\C}{\ensuremath{\mathbb{C}}}

\newcommand{\Q}{\mathbb{Q}}

\numberwithin{equation}{section}

\usepackage{lastpage}
\usepackage{fancyhdr}
\begin{document}

\title{The Explicit Sato-Tate Conjecture for Primes in Arithmetic Progressions}

\author{Trajan Hammonds}
\address{Department of Mathematics, Carnegie Mellon University, Pittsburgh, PA 15213}
\email{\textcolor{blue}{\href{mailto:thammond@andrew.cmu.edu}{thammond@andrew.cmu.edu}}}

\author{Casimir Kothari}
\address{Department of Mathematics, University of Chicago, Chicago, IL 60637}

\email{\textcolor{blue}{\href{mailto:ckothari@uchicago.edu}{ckothari@uchicago.edu}}}

\author{Noah Luntzlara}
\address{Department of Mathematics, University of Michigan, Ann Arbor, MI 48109}
\email{\textcolor{blue}{\href{mailto:nluntzla@umich.edu}{nluntzla@umich.edu}}}

\author{Steven J. Miller}
\address{Department of Mathematics and Statistics, Williams College, Williamstown, MA 01267}
\curraddr{Department of Mathematics, Carnegie Mellon University, Pittsburgh, PA 15213}

\email{\textcolor{blue}{\href{mailto:sjm1@williams.edu}{sjm1@williams.edu}},  \textcolor{blue}{\href{Steven.Miller.MC.96@aya.yale.edu}{Steven.Miller.MC.96@aya.yale.edu}}}

\author{Jesse Thorner}
\address{Department of Mathematics, Stanford University, Stanford, CA 94305}
\email{\textcolor{blue}{\href{mailto:jthorner@stanford.edu}{jthorner@stanford.edu}}}

\author{Hunter Wieman}
\address{Department of Mathematics and Statistics, Williams College, Williamstown, MA 01267}
\email{\textcolor{blue}{\href{mailto:hlw2@williams.edu}{hlw2@williams.edu}}}


\date{\today}

\begin{abstract} 

Let $\tau(n)$ be Ramanujan's tau function, defined by the discriminant modular form
$$\Delta(z) = q\prod_{j=1}^{\infty}(1-q^{j})^{24}\ =\ \sum_{n=1}^{\infty}\tau(n) q^n
\,,q=e^{2\pi i z}$$ (this is the unique holomorphic normalized
cuspidal newform of weight 12 and level 1).  Lehmer's conjecture asserts
that $\tau(n)\neq 0$ for all $n\geq 1$; since $\tau(n)$ is multiplicative,
it suffices to study primes $p$ for which $\tau(p)$ might possibly be
zero. Assuming standard conjectures for the twisted symmetric power
$L$-functions associated to $\tau$ (including GRH), we prove that if
$x\geq 10^{50}$, then $$\#\{x < p\leq 2x: \tau(p) = 0\} \leq 1.22 \times 10^{-5} \frac{x^{3/4}}{\sqrt{\log x}},$$ a substantial improvement on the implied constant in previous work. To achieve this, under the same hypotheses, we prove an explicit version of the Sato-Tate conjecture for primes in arithmetic progressions.
\end{abstract}

\maketitle


\section{Introduction} 

Let $q = e^{2\pi i z}$ with $\mathrm{Im}(z) > 0$, and let 
\[
f(z) = \sum_{n=1}^{\infty} a_f(n) q^n \in S_k^{\text{new}}(\Gamma_0(N))
\]
be a normalized cusp form of even weight $k \geq 2$ and level $N$ such that $f$ is an eigenform of all Hecke operators and of all Atkin-Lehner involutions $\mid_k W(N)$ and $\mid_k W(Q_p)$ for all $p \mid N$. We call such a cusp form a \textit{newform} (see \cite[Section 2.5]{Ono} for details). One implication of Deligne's proof of the Weil conjectures is that if $p$ is prime then there exists $\theta_p \in [0,\pi]$ such that
\[
a_f(p) = 2p^{(k-1)/2}\cos \theta_p.
\]
It is natural to consider the distribution of the angle $\theta_p$ in sub-intervals of $[0,\pi]$. The Sato-Tate conjecture, now a theorem due to Barnet-Lamb, Geraghty, Harris, and Taylor \cite{BLGHT}, gives us this distribution.  Let $\pi(x)$ for $x > 0$ denote the number of primes at most $x$ and $\rm{Li}(x)$ be the logarithmic integral of $x$. 
\begin{thm}[Sato-Tate Conjecture]
Let $f(z) \in S_k^{\rm{new}}(\Gamma_0(N))$ be a non-CM newform. If $F: [0,\pi] \rightarrow \C$ is a continuous function, then
$$\lim_{x \to \infty} \frac{1}{\pi(x)} \sum_{p \leq x} F(\theta_p) = \int_{0}^{\pi} F(\theta) \,d\mu_{\rm{ST}},$$
where $\,d\mu_{\rm{ST}} = \frac{2}{\pi} \sin^2 \theta \,d\theta$ is the Sato-Tate measure. Further, if we define
$$\pi_{f,I}(x) := \#\{p \leq x : \theta_p \in I\},$$
then we have
$$\pi_{f,I}(x) \sim \mu_{\rm{ST}}(I)\rm{Li}(x).$$
\end{thm}

The error term in the Sato-Tate Conjecture has been studied thoroughly under various hypotheses, including the cuspidality of the symmetric power lifts of the automorphic representation associated to $f$ and the generalized Riemann hypothesis (GRH) for the associated $L$-functions \cite{BK,VKM,RT}.  In particular, Rouse and Thorner \cite[Theorem 1.2]{RT} (under the aforementioned cuspidality and GRH assumptions) proved that
\[
|\pi_{f,I}(x)-\mu_{\mathrm{ST}}(I)\mathrm{Li}(x)|\leq 3.33 x^{3/4}-\frac{3x^{3/4}\log\log x}{\log x}+\frac{202x^{3/4}\log(N(k-1))}{\log x}
\]
for all $x\geq 2$, provided that $N$ is squarefree.  This saves a factor of $\sqrt{\log Nx}$ over the results in \cite{BK,VKM}. By weighing the primes with a smooth test function and taking $I=[\frac{\pi}{2}-\varepsilon,\frac{\pi}{2}+\varepsilon]$ (where $\varepsilon$ depends on $x$), Rouse and Thorner \cite[Theorem 1.3]{RT} also showed that
\begin{equation}
\label{eqn:RT_LT}
\#\{p\leq x\colon a_f(p)=0\}\leq \frac{9.3 x^{3/4}}{\sqrt{\log x}}+\textup{explicit lower-order terms,}\qquad x\geq 3.
\end{equation}
In the case where $f(z)=\Delta(z)$ is the newform of weight 12 and level 1 whose Fourier coefficients are given by the Ramanujan tau function $\tau(n)$, there is an important conjecture. 
\begin{conj}
\label{conj:lehmer}
	If $n\geq 1$, then $\tau(n)\neq 0$.  Equivalently, if $p$ is prime, then $\tau(p)\neq 0$.
\end{conj}

It appears Conjecture \ref{conj:lehmer} was first pondered seriously by Lehmer \cite{lehmer}.  
Serre \cite{SER2} observed that if $\tau(p)=0$, then $p=hM-1$, where $M=3094972416000$ and $h\geq 1$.  Moreover, $h+1$ is a quadratic residue modulo 23, and $h\equiv 0$, $30$, or $48\pmod{49}$.  This implies that if $\tau(p)=0$, then $p$ must lie in one of $33$ possible residue classes modulo $M=23\times49\times3094972416000$ (via the Chinese Remainder Theorem).  Moreover, using well-known congruences for $\tau(n)$ and the computation of the mod 11, mod 13, mod 17, and mod 19 Galois representations by Bosman \cite{Bosman}, we know that $\tau(n)\neq 0$ for $n<2.2798\ldots\times 10^{16}$.  Rouse and Thorner \cite{RT} used Bosman's work to prove that there are at most 1810 primes $p<10^{23}$ which satisfy Serre's conditions and for which $\tau_{12}(p) \equiv 0 \pmod{11\times 13\times 17\times19}$.

In this paper, we prove a variant of \eqref{eqn:RT_LT}, stated as Theorem \ref{thm:satotate}, where the primes are restricted to an arithmetic progression $a\pmod{q}$ with $\gcd(a,q)=1$.  This relies on standard conjectures regarding symmetric power $L$-functions, including their analytic continuation and the generalized Riemann hypothesis (GRH); see Conjecture \ref{kitchensink}.

Our interest in such a result lies in the choice of the arithmetic progression.  In particular, if $x$ is large, then Theorem \ref{thm:satotate} enables us to substantially decrease the implied constant in \eqref{eqn:RT_LT} via Serre's observation.  This leads to the following corollary, which is based on a standard conjecture about the behavior of the symmetric power $L$-functions.

\begin{cor}\label{cor:apsmall}
Assume Conjecture \ref{kitchensink} (which includes GRH and other standard analytic hypotheses) with $f=\Delta$. If $x \geq 1.2\times 10^{36}$, then
\[
\#\{x<p\leq 2x\colon \tau(p)=0\}
\]
is bounded by
\begin{equation*}
9.39\times 10^{-6}\ \frac{x^{3/4}}{\sqrt{\log x}}\ 
- 8.32\times 10^{-6}\ \frac{x^{3/4}\log\log x}{(\log x)^{3/2}} 
+ 3.56\times 10^{-4}\ \frac{x^{3/4}}{(\log x)^{3/2}} 
+ 58.15 \sqrt{x}\log x.
\end{equation*}
If $x\geq 10^{50}$, then
\begin{equation*}
\#\{x<p\leq 2x\colon \tau(p)=0\}\leq 1.22\times 10^{-5}\ \frac{x^{3/4}}{\sqrt{\log x}}.
\end{equation*}
\end{cor}

The rest of this paper is organized as follows. Section 2 gives an introduction to the analytic theory of symmetric power $L$-functions twisted by Dirichlet characters, details important assumptions in Conjecture \ref{kitchensink} and states the main result in Theorem \ref{thm:satotate}. Next, Section 3 gives the proofs of Theorem \ref{cor:apsmall} and Theorem \ref{thm:satotate}, assuming Proposition \ref{prop:3.3analogue}. In Section 4, we give the explicit formula and Section 5 proves a bound for the number of zeros on the critical line. Finally, in Section 6, we provide a proof of Proposition \ref{prop:3.3analogue}. We assume the reader is familiar with the standard results and notation. For reference see \cite{IK}.

%

\section{Symmetric power $L$-functions and the main result}

Let $k$, $q$ and $N$ be positive integers with $N$ squarefree, $k$ even, and $\gcd(N,q)=1$.  Let $f \in S_k^{\mathrm{new}}(\Gamma_0(N))$ be a non-$CM$ newform, and let $\chi$ be a primitive Dirichlet character with conductor $q$.  Our main object of study will be symmetric power $L$-functions of $f$ twisted by primitive Dirichlet characters $\chi$ of conductor $q$ satisfying $\gcd(N,q) = 1$. If we let $\alpha_p = e^{i \theta_p}$ and $\beta_p = e^{-i\theta_p}$ for $p \nmid N$, then the Dirichlet series associated to such an $L$-function is given by
\[
L(s, \sm) = \prod_{p \mid N} L_p(s, \sm) \prod_{p \nmid N} \prod_{j=0}^n (1-\alpha_p^j\beta_p^{n-j} \chi(p)p^{-s})^{-1}.
\]

We now assemble some standard desirable properties for the $L$-functions associated to twisted symmetric power $L$-functions.

\begin{conj}
\label{kitchensink}
Let $f$ and $\chi$ be as above. For each integer $n \geq 0$, the following are true.
\begin{enumerate}
\item The conductor of $L(s, \sm)$ is $q_{\sm} = N^nq^{n+1}$.
\item The equation of the gamma factor of $L(s, \sm)$ is 
\[
\gamma(s,\sm) = 
\begin{cases}
\prod_{j = 1}^{(n+1)/2} \Gamma_{\bb{C}}(s + (j - 1/2)(k-1) + \frak{a}) & \text{ if } n \text{ is odd,} \\ \\
\Gamma_{\bb{R}}(s + r) \prod_{j = 1}^{n/2} \Gamma_{\bb{C}}(s + j(k-1) + \frak{a}) & \text{ if } n \text{ is even}
\end{cases}
\]
where $\Gamma_{\bb{R}}(s) = \pi^{-s/2}\Gamma(s/2), \Gamma_{\bb{C}}(s) = 2(2\pi)^{-s} \Gamma(s), \ r = \frac{n}{2} \bmod 2$, and $\frak{a} = \frac{1-\chi(-1)}{2}$. ($\Gamma(s)$ denotes the usual gamma function.)
\item For each prime $p \mid N$, $L_p(s,\textup{Sym}^n f) = (1 - (-\lambda_pp^{1/2})^n \chi(p) p^{-s})^{-1}$, where $\lambda_p \in \{-1,1\}$ is an eigenvalue of the Atkin-Lehner operator $W(p)$ acting on the $f$.
\item Let
\[
\delta_{n, \chi}=\begin{cases}
	1&\mbox{if $\chi$ is trivial and $n=0$,}\\
	0&\mbox{otherwise.}
\end{cases}
\]
The completed $L$-function
\[
\Lambda(s, \sm):=(s(1-s))^{\delta_{\chi,n}}(q_{\sm})^{\frac{s}{2}} \gamma(s, \sm) L(s, \sm)
\]
is an entire function of order 1.
\item There exists a complex number $\epsilon_{\sm}$ of modulus 1 such that for all $s\in\mathbb{C}$, we have $\Lambda(s, \sm) = \epsilon_{\sm} \Lambda(1-s,\mathrm{Sym}^n f\otimes \bar{\chi})$.
\item The Generalized Riemann Hypothesis (GRH):  Each zero of $\Lambda(s, \sm)$ has real part equal to $\frac{1}{2}$.
\end{enumerate}

\end{conj}

\begin{rek}
Since the initial submission of this article, it has been shown that there exists a cuspidal automorphic representation of $\mathrm{GL}_{n+1}(\mathbb{A}_{\Q})$ whose $L$-function equals $L(s, \mathrm{Sym}^n f)$ (apart from at most finitely many Euler factors) for all $n \geq 1$ (see \cite{NT1}, \cite{NT2}); this implies Parts (1)-(5) in Conjecture \ref{kitchensink} for all $n \geq 1$. 
\end{rek}

We now state our main result, an explicit version of the Sato-Tate conjecture for primes in an arithmetic progression. 
\begin{thm}
\label{thm:satotate}
Let $f(z) = \sum_{n=1}^{\infty} a_f(n) q^n \in S_{k}^{\rm{new}}(\Gamma_0(N))$ be a
 newform that satisfies Conjecture \ref{kitchensink}, and let $\phi(t)$ be an infinitely differentiable smooth nonnegative test function with compact support satisfying $\phi(t)\leq 2$, $\mbox{supp}(\phi)\subset [\frac{1}{2},\frac{5}{2}]$. Let $\phi_x(t) = \phi(t/x)$, let $\Phi(s)$ be the Mellin transform of $\phi$, and let $I = [\alpha, \beta] \subset [0,\pi]$.  Define
 \[
 C_n(\phi) = \frac{1}{(2\pi)^{n-1}}\int_{-\infty}^\infty \big|\phi^{(n)}(e^{2\pi t})e^{(2n+1)\pi t}\big|dt.
 \]
 If $x \geq \max\{4.6\times 10^7, \hspace{.05cm} 7500\big(\varphi(q)\log \varphi(q)\big)^2\}$
then 
\begin{align}
&\Bigg| \sum_{\substack{p \nmid N \\ \theta_p\in I \\ p\equiv a\hspace{-2mm}\pmod{q}}}\log(p)\phi_x(p)
-\frac{x}{\varphi(q)} \mu_{ST}(I) \cdot \int_{-\infty}^{\infty}\phi(t)dt \Bigg| \nonumber \\
&\leq \frac{x^{3/4}}{\sqrt{\varphi(q)\log x}}
\Bigg(2\Phi(1) \log x + \Big(\sqrt{C_0(\phi)C_2(\phi)} + 3C_0(\phi)\Big)
\big(1.31 \log x - 2.61 \log \log x \big) \nonumber \\
&+ 16.21 + 378.9 C_0(\phi) + 100.9 C_2(\phi) + \log(Nq(k-1))\Big( 28.37 C_0(\phi) + 4.06 C_2(\phi) + .001\Big)\Bigg).
\end{align}
\end{thm}

\section{Proofs of Corollary \ref{cor:apsmall} and Theorem \ref{thm:satotate}}

\subsection{Fourier Decomposition of the Indicator Function}

In order to make the sum in Theorem \ref{thm:satotate} more tractable, we would like to approximate an indicator function for $\theta_p \in I \subset [0,\pi]$.
Let $M$ be a positive integer, $I = [\alpha, \beta] \subset [0, \pi]$, and $U_n(x)$ be the $n$-th Chebyshev polynomial of the second type defined by
\begin{equation} \label{eqn:cheby}
    U_n(\cos \theta) = \frac{\sin\left((n+1)\theta\right)}{\sin \theta}.
\end{equation} Lemma 3.1 of \cite{RT} states that there exist trigonometric polynomials 
$$F^\pm_{I,M}(\theta) = \sum_{n=0}^M \hat{F}^\pm_{I,M}(n)U_n(\cos\theta)$$
which satisfy $\forall x \in [0, \pi],$ 
\be \label{lem:trigbounds}
F^-_{I,M}(x) \leq \mathbf{1}_{I}(x) \le F^+_{I,M}(x),
\ee
\be \label{eqn:zerofc}
|\hat{F}^\pm_{I,M}(0) - \mu_{ST}(I)| \leq \frac{4}{M+1},
\ee
and
\be
|\hat{F}^\pm_{I,M}(n)|\leq \frac{4}{M+1}+\frac{4}{\pi n},
\ee
\\where $\mathbf{1}_{I}$ is the indicator function for the interval $I$. Additionally, we have the following lemma.

\begin{lem} \label{lem:indfcn}
Assume $M \geq 8$ and let $C = 32(1/3 + 1/\pi)$.  Then the following inequalities hold:
\begin{align}
\sum_{n = 1}^M |\hat{F}^\pm_{I,M}(n)| &\leq \frac{2}{\pi}\log M + \frac{21}{5}\\
\sum_{n = 1}^M n|\hat{F}^\pm_{I,M}(n)| &\leq \frac{CM}{16} \\
\sum_{n = 1}^M (n+1)|\hat{F}^\pm_{I,M}(n)| &\leq \frac{CM}{16} + \frac{2}{\pi}\log M + \pi.
\end{align}
\end{lem}

\begin{proof}
The desired bounds for $F^-_{I,M}(\theta)$ are proved in \cite[Lemma 5.1]{CPS}; the bounds for $F^+_{I,M}(\theta)$ are proved similarly.
\end{proof}

\subsection{Proof of Theorem \ref{thm:satotate}} 

Consider the Fourier expansion
\be \label{eqn:firstDecomp}
\sum_{\substack{p \nmid N \\ \theta_p \in I \\ p \equiv a \hspace{-2mm}\pmod q}} \phi_x(p) \log p = \frac{1}{\varphi(q)}\sum_{\chi(q)} \overline{\chi}(a) \sum_{p \nmid N} \mathbf{1}_{I}(\theta_p) \chi(p) \phi_x(p)\log p
\ee
of the sum in Theorem \ref{thm:satotate}. It will later become convenient to instead consider this sum over primitive characters, hence, we introduce the following lemma which bounds the error from passing to a sum over primitive characters. 

\begin{lem} \label{lem:primchars}
If $\chi$ is a Dirichlet character modulo $q$ induced by the primitive Dirichlet character $\chi'$, then  
\begin{equation}
\Big|\sum_{\substack{p \nmid N \\ \theta_p \in I \\ p \equiv a\hspace{-2mm}\pmod{q}}} \phi_x(p) \log p - \frac{1}{\varphi(q)}\sum_{\chi(q)} \overline{\chi}(a) \sum_{p \nmid N} \mathbf{1}_{I}(\theta_p) \chi'(p) \phi_x(p)\log p \Big| \leq \mxp \log q \nonumber.
\end{equation}
\end{lem}

\begin{proof}
The two terms differ only at $p|q$, where the contribution from the first term is zero, and the contribution from the second term is bounded in absolute value by $\phi_x(p)\log p$. Therefore

\begin{align*}
\Big|\sum_{\substack{p \nmid N \\ \theta_p \in I \\ p \equiv a\hspace{-2mm}\pmod{q}}} \phi_x(p) \log p - \frac{1}{\varphi(q)}\sum_{\chi(q)} \overline{\chi}(a) \sum_{p \nmid N} \mathbf{1}_{I}(\theta_p) \chi'(p) \phi_x(p)\log p \Big| 
&\leq \sum_{p | q} \phi_x(p) \log p.
\end{align*}
The result now follows.
\end{proof}
Before we prove Theorem \ref{thm:satotate}, we first give a useful preliminary bound.
\begin{lem} \label{lem:satotateerror}
Let $I = [a,b] \subset [0,\pi]$ be a subinterval, and let $M \geq 2$. Then 
\begin{equation*}
\Big| \sum_{\substack{p \nmid N \\ \theta_p\in I \\ p\equiv a \bmod q}}\phi_x(p)\log(p)
-
\frac{x}{\varphi(q)} \mu_{ST}(I) \int_{-\infty}^{\infty}\phi(t)dt \Big|
\end{equation*}
is bounded above by 
\begin{align*}
\frac{\Phi(1) x}{\varphi(q)}\cdot\frac{4}{M} &+ \frac{1}{\varphi(q)} \sum_{\chi(q)}\sum_{n = 0}^M |\hat{F}^\pm_{I,M}(n)|\Big|\sum_{p \nmid N} U_n(\cos \theta_p) \chi'(p) \phi_x(p) \log(p) - \delta_{n,\chi}\Phi(1) x\Big| \\&+ \max_{t \in \bb{R}} |\phi(t)| \log q.
\end{align*}

\end{lem}
\begin{proof}
	By Lemma \ref{lem:primchars}, we have
\begin{align*}
    &\Big| \sum_{\substack{p \nmid N \\ \theta_p\in I \\ p\equiv a \bmod q}}\phi_x(p)\log(p)
-
\frac{x}{\varphi(q)} \mu_{ST}(I) \int_{-\infty}^{\infty}\phi(t)dt \Big| \\
\leq \Big|\frac{1}{\varphi(q)}\sum_{\chi(q)}\overline{\chi}(a)&\sum_{p \nmid N} \mathbf{1}_I(\theta_p)\chi'(p)\phi_x(p)\log(p) - \frac{x}{\varphi(q)}\mu_{ST}(I)\int_{-\infty}^{\infty} \phi(t) \ dt\Big| + 
\max_{t \in \bb{R}} |\phi(t)| \log q .
\end{align*}
Next we use \eqref{lem:trigbounds} and \eqref{eqn:zerofc} to deduce
\begin{align*}
&\Big|\sum_{p \nmid N} \mathbf{1}_I(\theta_p)\frac{1}{\varphi(q)}\sum_{\chi(q)}\overline{\chi}(a)\chi'(p)\phi_x(p)\log p - \frac{x}{\varphi(q)}\mu_{ST}(I)\int_{-\infty}^{\infty} \phi(t) \,dt\Big| \\
&\leq \max_{\pm}\Big|\sum_{p \nmid N} F^{\pm}_{I,M}(\theta_p)\frac{1}{\varphi(q)}\sum_{\chi(q)}\overline{\chi}(a)\chi'(p)\phi_x(p)\log p - \frac{x}{\varphi(q)}\mu_{ST}(I)\int_{-\infty}^{\infty} \phi(t) \,dt\Big|\\
&\leq \frac{\Phi(1) x}{\varphi(q)}\cdot\frac{4}{M} + \frac{1}{\varphi(q)} \sum_{\chi(q)}\sum_{n = 0}^M |\hat{F}^\pm_{I,M}(n)|\Big|\sum_{p \nmid N} U_n(\cos \theta_p) \log(p) \chi'(p) \phi_x(p) - \delta_{n,\chi}\Phi(1) x\Big|,
\end{align*}
as desired.
\end{proof}


Now, we define
\begin{equation}
\label{eqn:Cn_def}
C_n(\phi) = \frac{1}{(2\pi)^{n-1}}\int_{-\infty}^\infty \big|\phi^{(n)}(e^{2\pi t})e^{(2n+1)\pi t}\big|dt.
\end{equation}
Theorem \ref{thm:satotate} then follows from the following proposition, which we prove in Section \ref{sec:propproof}.

\begin{prop} \label{prop:3.3analogue}
Assume the hypotheses of Theorem \ref{thm:satotate}.  If $n \geq 1$, then 
\begin{align}
&\Big|\sum_{p \nmid N} U_n(\cos\theta_p)\log(p)\chi(p)\phi_x(p)-\delta_{n,\chi} \Phi(1) x\Big| \nonumber \\
&\leq 2\sqrt{x}\Bigg(\Big(\sqrt{C_0(\phi)C_2(\phi)} + 3C_0(\phi)\Big)\bigg((n+8)\log(n) + (n+1)\Big(\frac{1}{7} + \log(Nq(k-1))\Big)\nonumber \\ &+ \frac{9}{2} 
+ \frac{36}{n}\bigg) + \sqrt{C_0(\phi)C_2(\phi)}\bigg(\frac{n}{2} + 7 + \frac{24}{n}\bigg) + C_2(\phi)\bigg(1 + \frac{8}{n}\bigg) + 2(n+1)\Bigg) \nonumber \\
&+ 2(n+1)\log N.
\end{align}
Additionally, a bound for the case when $n = 0$ is given by (\ref{eqn:n0case}).
\end{prop}

We now prove Theorem \ref{thm:satotate} assuming Proposition \ref{prop:3.3analogue}.
\begin{proof}[Proof of Theorem \ref{thm:satotate}] Choose $M = 2x^{1/4}/\sqrt{\varphi(q)\log x}$. We first show that when $x\geq \max\{4.6\times 10^7,\hspace{.05cm} 7500(\varphi(q)\log \varphi(q))^2\}$, $M \geq 8$. For all $\varphi(q) \leq 24$, the bound follows by direct computation with $x\geq 4.6\times 10^7$. Otherwise, we have that $x\geq 7500(\varphi(q)\log \varphi(q))^2$, and therefore 
\be \label{eqn:Mbound}
M \geq \frac{18.61\sqrt{\log \varphi(q)}}{\sqrt{8.93 + 2\log \varphi(q) + 2\log \log \varphi(q)}}.
\ee
This expression evaluates to $8.006$ for $\varphi(q) = 28$, noting that $\varphi(q)$ will never take on the values $25,26,$ and $27$. Because this lower bound is increasing in $\varphi(q)$ and $M$ is increasing in $x$, it follows that for all $x\geq \max\{4.6\times 10^7,\hspace{.05cm} 7500(\varphi(q)\log \varphi(q))^2\}$, $\varphi(q)$, we have $M\geq 8$, allowing us to apply Lemma \ref{lem:indfcn}.

Next we substitute Proposition \ref{prop:3.3analogue} into the inner sum in the bound from Lemma \ref{lem:satotateerror}. We can then apply Lemma \ref{lem:indfcn} and equation (3.2) to bound the resulting sum. We also use $\log M + 1$ as an upper bound for the $M^{\text{th}}$ harmonic sum, $\pi^2/6$ as an upper bound for $\sum_{n=1}^M\frac{1}{n^2}$, and $\log M$ as an upper bound for $\log n$. This gives
\begin{align*}
&\sum_{n = 0}^M |\hat{F}^\pm_{I,M}(n)|\Big|\sum_{p \nmid N} U_n(\cos \theta_p) \log(p) \chi'(p) \phi_x(p) - \delta_{n,\chi}\Phi(1) x\Big| \\
&\leq \Big(\frac{CM}{16} + \frac{2\log M}{\pi} + \pi\Big) \bigg[2\sqrt{x}\Big(\big(\sqrt{C_0(\phi)C_2(\phi)} + 3C_0(\phi)\big)\\&\cdot(\log M + \frac{1}{7} + \log(Nq(k-1))) 
+ \frac{\sqrt{C_0(\phi)C_2(\phi)}}{2} + 2\Big) + 2\log N\bigg] + \bigg(\frac{2}{\pi}\cdot\log M + \frac{21}{5}\bigg)
\\ &\cdot2\sqrt{x}\bigg[\Big(\sqrt{C_0(\phi)C_2(\phi)} + 3C_0(\phi)\Big)
\cdot\Big(7\log M + \frac{9}{2}\Big)+ \frac{13}{2}\sqrt{C_0(\phi)C_2(\phi)} + C_2(\phi)\bigg] \\&+ \bigg(\frac{4}{M}\big(\log M + 1\big) + \frac{4}{\pi}\cdot\frac{\pi^2}{6}\bigg) 
\cdot 2\sqrt{x}\Big(60\sqrt{C_0(\phi)C_2(\phi)} + 108C_0(\phi) + 8C_2(\phi)\Big) \\
&+ \sqrt{x}\Big(\sqrt{C_0(\phi)C_2(\phi)}\big(42.96 + 2\log(Nq(k-1))\big) + \big(72.8 + 6\log(Nq(k-1))\big)C_0(\phi) \nonumber \\
&+ 23.56C_2(\phi) + 3.983 \Big) + 2\log N.
\end{align*}

We observe that the first product in this bound gives some terms of order $\sqrt{x}M$ and some terms of order $\sqrt{x}M\log M$. 
Substituting in $M = 2x^{1/4}/\sqrt{\varphi(q)\log x}$ will give their contributions to the final bound. 
We next bound all the remaining lower order terms by terms of order $x^{3/4}/\sqrt{\varphi(q)\log x}$. 
We replace instances of $1/M$ with $1/8$ and $\log(M)/M$ with $
\log(8)/8$, and then multiply all the constant terms by $\sqrt{x}/\sqrt{4.7\cdot 10^7}$. The remaining lower order terms are all of order $\sqrt{x}(\log M)^{\ell}, \ell \in\{0,1,2\}$. 
Let
\be
L = \max\{4.6\times 10^7, \hspace{.05cm} 7500\big(\varphi(q)\log \varphi(q)\big)^2\} \nonumber
\ee
be the lower bound on $x$ as in Theorem \ref{thm:satotate}, and let
\be
M_L = 2L^{1/4}/\sqrt{\varphi(q)\log L} \nonumber
\ee
be $M$ evaluated at $x = L$. Because $x^{1/4}/(\sqrt{\log x}(\log M)^{\ell})$ is increasing in $x$, we have that $\sqrt{x}(\log M)^{\ell}$ is bounded by
\be
\sqrt{x}(\log M)^{\ell}\Bigg(\frac{x^{1/4}/(\sqrt{\log x}(\log M)^{\ell})}{L^{1/4}/(\sqrt{\log L}(\log M_L)^{\ell})}\Bigg) = \frac{x^{3/4}}{\sqrt{\varphi(q)\log x}}\Bigg(\frac{\sqrt{\varphi(q)\log L}}{L^{1/4}}\Bigg)(\log M_L)^{\ell}. \nonumber
\ee
A simple calculation gives that for all $\varphi(q)$, $\sqrt{\varphi(q)\log L}/L^{1/4} \leq .2499$ (it achieves this value when $\varphi(q) = 24$) and that for all $\varphi(q) \geq 24, (\log M_L)^{\ell} \leq 2.578^{\ell}$.\footnote{2.578 upper bounds the limit of the right hand side of (\ref{eqn:Mbound}) as $\varphi(q)$ goes to infinity.} Lastly, we observe that since $L$ is fixed for $\varphi(q) \leq 24$, $(\sqrt{\varphi(q)\log L}/L^{1/4})(\log M_L)^{\ell}$ is increasing in $\varphi(q)$ over this domain. Therefore any bound on $(\sqrt{\varphi(q)\log L}/L^{1/4})(\log M_L)^{\ell}$ for $\varphi(q) = 24$ will also suffice for $\varphi(q) < 24$. Consequently, we have that
\be
\sqrt{x}(\log M)^{\ell} \leq \frac{x^{3/4}}{\sqrt{\varphi(q)\log x}}(.2499)(2.578)^{\ell}. \nonumber
\ee
This gives the contributions to our final bound from the sum in the bound of Lemma \ref{lem:satotateerror}. 

A similar argument gives  
\be
\max_{t \in \bb{R}} |\phi(t)| \log q \leq .00025 \frac{x^{3/4}}{\sqrt{\varphi(q)\log x}}. \nonumber
\ee
Lastly, we observe that 
\be
\frac{\Phi(1) x}{\varphi(q)}\cdot\frac{4}{M} = 2\Phi(1)\frac{x^{3/4}\sqrt{\log x}}{\sqrt{\varphi(q)}}, \nonumber
\ee
and collecting these terms gives the desired bound.
\end{proof}

\subsection{Proof of Corollary \ref{cor:apsmall}}
We now prove Corollary \ref{cor:apsmall} by introducing some additional results. We make the choice of test function $\phi$ as 
\be \phi(y) = 
\begin{cases}
\exp(\frac{4}{3} + \frac{1}{(y - 1/2)(y - 5/2)}) & \text{ if } 1/2 < y < 5/2\\
0 & \text{ otherwise, }
\end{cases}
\ee
which is a pointwise upper bound for the indicator function for $[1, 2]$. As in \cite{RT}, we define $I_M = [\frac{\pi}{2}(1 - \frac{1}{2M}), \frac{\pi}{2}(1 + \frac{1}{2M})]$, and note that $\mu_{ST}(I_M)<\frac{1}{M} $.  Note that $\frac{\log p}{\log x} \geq 1$ for $x < p \leq 2x$ and that $\Phi(1)\leq 1.684$.  Direct substitution into  the bound of Lemma \ref{lem:satotateerror} yields the following lemma.

\begin{lem}\label{lem:dyadicpi}
If $M$ is a positive integer, then
\begin{equation}
\label{eqn:prime_count_APs_dyadic}
\#\{x<p\leq 2x\colon p\equiv a\hspace{-2mm}\pmod{q},~\theta_p\in I_M\}
\end{equation}
is bounded above by
{\small\begin{align*}
&\frac{5\Phi(1)x}{\varphi(q)M\log x}+\frac{1}{\varphi(q) \log x} \sum_{n \leq M} |\hat{F}^+_{I,M}(n)|\sum_{\substack{\chi \hspace{-2mm}\mod q \\ \chi' \textup{ induces } \chi}} \Big| \sum_p U_{n} (\cos \theta_p) \phi_x(p) \chi'(p) \log(p) - \delta_{n,\chi} \Phi(1) x \Big|\\
&+ \frac{2\log q}{\log x}.
\end{align*}}%

\end{lem}

Given this choice of $\phi$, we compute the constants $C_0(\phi)$ and $C_2(\phi)$ as defined in equation (\ref{eqn:Cn_def}) and use Proposition \ref{prop:3.3analogue} to prove Corollary \ref{cor:apsmall}. 

\begin{proof}[Proof of Corollary \ref{cor:apsmall}]
 Let $f(z) = \Delta(z)$ denote the discriminant modular form. By the work of Serre \cite{SER2}, if $\tau(p) = 0$ then $p$ is in one of 33 possible residue classes modulo $q = 23 \times 49 \times 3094972416000.$ Thus we have $N = 1, k = 12,$ and $q = 23 \times 49 \times 3094972416000.$

Assume first $x \geq 1.2\times 10^{36}$, and pick $M = 7.0\times 10^{-8}x^{1/4}/\sqrt{\log x}$, so that in particular we have $M\geq 8$.  We can then apply the bound given in Proposition \ref{prop:3.3analogue} to bound the inner sum in Lemma \ref{lem:dyadicpi}. Summing over $n$, we obtain that \eqref{eqn:prime_count_APs_dyadic} is bounded by 
{\small\begin{align*}
2.843\times 10^{-7}\ \frac{x^{3/4}}{\sqrt{\log x}}\ - 
&2.524\times 10^{-7}\ \frac{x^{3/4}\log\log x}{(\log x)^{3/2}} +
1.076\times 10^{-5}\ \frac{x^{3/4}}{(\log x)^{3/2}} + 1.762 \sqrt{x}\log x.
\end{align*}}%
\noindent Since $\pi/2 \in I_M$, this is an upper bound on $\# \{x < p \leq 2x\colon \tau(p) = 0,~p \equiv a \, \mathrm{mod} \, q \}$. We then multiply by 33 to get the first bound in Corollary \ref{cor:apsmall}. When $x\geq 10^{50}$, we absorb the lower order terms into the leading term and obtain a bound of
\be
1.22\times 10^{-5}\ \frac{x^{3/4}}{\sqrt{\log x}}
\ee
completing the proof.
\end{proof}

\section{The Mellin Transform}
In this section we obtain an explicit formula for $\psi_{\sm}$ by pushing a contour integral and evaluating contributions from the residues and zeros. We define the numbers $\Lambda_{\sm}(j)$ by 
\[
-\frac{L'}{L}(s, \sm) = \sum_{j=1}^{\infty} \frac{\Lambda_{\sm}(j)}{j^s}, \ \rm{Re}(s) > 1.
\]
Let $U_n$ be the $n$-th Chebyshev polynomial of the second type as in (\ref{eqn:cheby}).
A simple computation shows that for any integer $j$, we have that 
\begin{equation}
    \Lambda_{\sm}(j) = 
\begin{cases} 
U_n(\cos(m\theta_p))\chi(j)\log(p) & \text{if } j = p^m \text{ for some } p\nmid N \text{ and } m \geq 1, \\
t_{m,n,p}p^{-mn/2}\log(p) & \text{if } j = p^m \text{ for some } p\mid N \text{ and } m \geq 1, \\
0 & \text{ otherwise,}
\end{cases}
\end{equation}
where $|t_{m,n,p}| \leq 1$.  We observe via inversion that
\be \label{eqn:contour}
\sum_{n} \Lambda_{\sm}(n)\phi_x(n) = \frac{1}{2 \pi i} \int_{2-
i\infty}^{2 + i\infty} - \frac{L'}{L}(s, \sm) \Phi_x(s) ds.
\ee

Then, by pushing the contour from (\ref{eqn:contour}) to negative infinity and accounting for residues as in the proof of Lemma 3.3 of \cite{Ro}, we can rewrite this integral as a sum over the zeros of $L(s,\sm)$: 
\be
\sum_{n} \Lambda_{\sm}(n)\phi_x(n) = \delta_{n, \chi} \Phi_x(1) - \sum_{\rho} \Phi_x(\rho).
\ee
The $\delta_{n, \chi} \Phi_x(1)$ term results from the residue of order $1$ at $s=1$, which only occurs for the $0$-th power symmetric $L$-function twisted by the trivial character.



\section{Bounding the number of zeros on the critical line}

Recall the definition of $\Lambda(s,\mathrm{Sym}^n f\otimes\chi)$ in Conjecture \ref{kitchensink}.  By the Hadamard factorization theorem, there exist constants $a_{\sm}$ and $b_{\sm}$ such that 
\[
\Lambda(s, \sm) = e^{a_{\sm}+b_{\sm}s} \prod_{\rho}\Big(1 - \frac{s}{\rho}\Big)e^{s/\rho},
\]
where $\rho$ ranges over the zeros of $\Lambda(s,\sm)$. After taking the logarithmic derivative of each side, we obtain the identity 
\begin{equation}
 -\frac{L'}{L}(s, \sm) =\frac{1}{2} \log(N^n q^{n+1}) + \frac{\gamma'}{\gamma}(s, \sm) -b_{\sm}  -\sum_{\rho} \Big(\frac{1}{s - \rho} + \frac{1}{\rho}\Big).
\end{equation}

Before producing a bound, we establish the following lemmas.

\begin{lem} \label{lem:logderivbound}
If $s = \sigma + it$ and $\sigma > 1$, then 
\be
\left| \frac{L'}{L}(s, \sm)\right| \leq -(n+1) \frac{\zeta'}{\zeta}(\sigma).
\ee
\end{lem}

\begin{proof}
Since $|\Lambda_{\sm}(x)| \leq (n+1) \Lambda(x)$ and $|\chi(x)| \leq 1$ we have 
\be
\left| \frac{L'}{L}(s, \sm)\right| \leq \sum_{j=1}^{\infty} \left|\frac{\Lambda_{\sm}(j)}{j^s}\right| \leq  (n+1) \sum_{j = 1}^\infty \frac{\Lambda(j)}{j^{\sigma}} \leq -(n+1) \frac{\zeta'}{\zeta}(\sigma). \nonumber
\ee
\end{proof}

\begin{lem}\label{lem:5.3analogue}
If $\Re(s)\geq 2$ and $\Im(s) = T$ then 
\be
\Re\Big(\frac{\gamma'}{\gamma}(s, \sm)\Big) \leq \frac{n+1}{2}\Big(\log(k-1) + \log(n+|T|+3)-1\Big) +\frac{7}{2}\log(n+|T|+3).
\ee
\end{lem}

\begin{proof}
In Lemma 5.3 of \cite{RT}, the above bound is proven for the gamma factors of $L(s, \text{Sym}^nf)$.  However, the assumed form of the gamma factors of $L(s, \text{Sym}^nf)$ differs from our gamma factors only in the real parts of the inputs (see Conjecture 1.1 of \cite{RT}).  Note, however that the above bound does not rely on $\Re(s)$, except that it be at least 2.  Hence, the bound follows immediately from Lemma 5.3 of \cite{RT}.
\end{proof}

\noindent We are now ready to obtain a bound for the vertical distribution of zeros.

\begin{thm} \label{thm:n(T)}
Let $\mathfrak{n}_{\sm}(T) = \#\{\rho = 1/2 + i\gamma : L(\rho,\sm)=0, T \leq \gamma \leq T+3\}$.  Then
\be
\mathfrak{n}_{\sm}(T) \leq \frac{3(n+1)}{2}\Big(\log\big(Nq(k-1)\big) + \log(n+|T|+9/2)+\frac{1}{7}\Big) +\frac{21}{2}\log(n+|T|+9/2).
\ee
\end{thm}

\begin{proof}
Fix $s_0 = 2 + i(T+3/2)$.  Following the arguments in Lemma 5.4 of \cite{RT}, we have that
\be
\sum_{\rho} \Re\Big(\frac{1}{s_0 - \rho} \Big) = \frac{1}{2}\log(N^n q^{n+1}) + \Re\Big(\frac{\gamma'}{\gamma}(s_0, \sm)\Big)  + \Re\Big( \frac{L'}{L}(s_0, \sm) \Big),
\ee
where the sum is over the nontrivial zeros $\rho$ of $L(s, \sm)$.
\noindent We first note that $\frac{1}{2}\log\left(N^nq^{n+1}\right) \leq \frac{n+1}{2}\log\left(Nq\right)$. Next we note $$\Re\Big( \frac{L'}{L}(s_0, \sm) \Big) \leq \left|\frac{L'}{L}(s_0, \sm)\right| \leq -(n+1)\frac{\zeta'}{\zeta}(2) \leq \frac{n+1}{2}(1.14)$$
by Lemma \ref{lem:logderivbound} and a direct computation. Summing these estimates with 
Lemma \ref{lem:5.3analogue}, we obtain
\be
\sum_{\rho} \Re\Big(\frac{1}{s_0 - \rho} \Big)  \leq \frac{(n+1)}{2}\Big(\log\big(Nq(k-1)\big) + \log(n+|T|+9/2)+\frac{1}{7}\Big) +\frac{7}{2}\log(n+|T|+9/2).
\ee
Lastly, we note that for all $\rho$ satisfying $T \leq \gamma \leq T+3 $
$$\Re\bigg(\frac{1}{s_0 - \rho}\bigg) = \Re\bigg(\frac{1}{\frac{3}{2} + i(T+\frac{3}{2} - \gamma)}\bigg) = \frac{\frac{3}{2}}{\frac{3}{2}^2 + (T-\gamma +\frac{3}{2})^2} \geq \frac{1}{3}.$$
\noindent Thus it follows that
\begin{align}
&\mathfrak{n}_{\sm}(T) \leq 3\sum_{\rho} \Re\bigg(\frac{1}{s_0 - \rho}\bigg) \nonumber \\
&\leq \frac{3(n+1)}{2}\Big(\log\big(Nq(k-1)\big) + \log(n+|T|+9/2)+\frac{1}{7}\Big) +\frac{21}{2}\log(n+|T|+9/2)
\end{align}
as desired.
\end{proof}

\section{Explicit Formula} \label{sec:expform}

We have thus shown the explicit formula
\[
\sum_{n} \Lambda_{\sm}(n)\phi_x(n) = \delta_{n, \chi} \Phi_x(1) - \sum_{\rho} \Phi_x(\rho),
\]
where the sum is over the zeros of $L(s, \sm)$. We now proceed to obtain an upper bound on the sum over zeros and use this to complete the proof of Proposition \ref{prop:3.3analogue}.

\subsection{Preliminaries}

Let $\rho = \frac{1}{2} + i\gamma$ denote a nontrivial zero of $L(s, \sm)$. Then, the following lemma gives a useful upper bound on $|\Phi_x(\rho)|$.

\begin{lem} \label{lem:phirho}
We have
\be 
|\Phi_x(\rho)| \leq \sqrt{x}\min\bigg\{C_0(\phi), \frac{C_1(\phi)}{|\gamma|}, \frac{C_2(\phi)}{|\gamma|^2},\dots \bigg\}.
\ee
\end{lem}

\begin{proof}
From \cite{Ro}, we have 
\be 
|\Phi_x(\rho)|=\sqrt{x}\big|\hat{h}(-\gamma)\big|,
\ee
where $h(t) = 2\pi\phi(e^{2\pi t})e^{\pi t}$. Trivially, we have
\be
\big|\hat{h}(-\gamma)\big| = \Big|2\pi\int_{-\infty}^\infty \phi(e^{2\pi t})e^{(1+2i\gamma)\pi t}dt\Big| \le 2\pi\int_{-\infty}^\infty \big|\phi(e^{2\pi t})e^{\pi t}\big|dt = C_0(\phi),
\ee
and integrating by parts $n$ times establishes
\be
|\hat{h}(-\gamma)| \leq \Big|\frac{1}{(2\pi)^{n-1} \gamma^n}\Big| \int_{-\infty}^\infty \big|\phi^{(n)}(e^{2\pi t})e^{(2n+1)\pi t}\big|dt = \frac{C_n(\phi)}{|\gamma|^n}
\ee
as desired.
\end{proof}

\subsection{Bounding the Sum Over Zeros}
We first use Theorem \ref{thm:n(T)} and Lemma \ref{lem:phirho}  to estimate $\sum_\rho\Phi_x(\rho)$ for nontrivial $\rho$.

\begin{prop}\label{prop:nontriv}
For $n\geq 1$, we have that $\left|\sum_{\rho} \Phi_x(\rho)\right|$ is bounded above by
\begin{align}
&\leq 2\sqrt{x}\Bigg(\Big(\sqrt{C_0(\phi)C_2(\phi)} + 3C_0(\phi)\Big)\bigg((n+8)\log(n) + (n+1)\Big(\frac{1}{7} + \log(Nq(k-1))\Big) \nonumber \\
&+ \frac{9}{2} + \frac{36}{n}\bigg) + \sqrt{C_0(\phi)C_2(\phi)}\bigg(\frac{n}{2} + 7 + \frac{24}{n}\bigg) + C_2(\phi)\bigg(1 + \frac{8}{n}\bigg) \Bigg),
\end{align}
where the sum is over the nontrivial zeros of $L(s, \sm)$.
For $n=0$, $\left|\sum_{\rho} \Phi_x(\rho)\right|$ is bounded above by
\begin{equation}
\sqrt{x}\Big(\sqrt{C_0(\phi)C_2(\phi)}\big(42.96 + 2\log(Nq(k-1))\big) + \big(72.8 + 6\log(Nq(k-1))\big)C_0(\phi) + 23.56C_2(\phi) \Big).
\end{equation}
\end{prop}

\begin{proof}
By the triangle inequality and the proof of Lemma \ref{lem:phirho}, we have
\begin{align}
&\left|\sum_\rho\Phi_x(\rho)\right| \le 2\sqrt{x}\sum_{j=0}^\infty \mathfrak{n}_{\sm}(3j)\big|\hat{h}(-j)\big| \nonumber \\  &\leq2\sqrt{x}\Bigg(\sum_{j=0}^{U/3}\mathfrak{n}_{\sm}(3j)C_0(\phi) + \mathfrak{n}_{\sm}(U)C_0(\phi) + \sum_{j\geq\frac{U}{3}+1} \mathfrak{n}_{\sm}(3j)\frac{C_2(\phi)}{(3j)^2} \Bigg) \nonumber \\
&\leq2\sqrt{x}\Bigg(\int_0^{U/3}\mathfrak{n}_{\sm}(3t)C_0(\phi)dt + 2\mathfrak{n}_{\sm}(U)C_0(\phi) + \int_{U/3}^\infty \mathfrak{n}_{\sm}(3t)\frac{C_2(\phi)}{(3t)^2}dt  \Bigg). \nonumber
\end{align} 
For brevity, we let $K_1 = \frac{3(n+1)}{2}\big(\log(Nq(k-1))+\frac{1}{7}\big)$ and $K_2 = \frac{3}{2}n+12$. We can now write $$\mathfrak{n}_{\sm}(j) \leq K_1 + K_2\log(n+j + 9/2).$$ 
With this notation, the above two integrals are bounded by
{\small\begin{align} \label{eqn:intEval}
&\Bigg(K_1 \Big(C_0(\phi)\frac{U}{3}+ C_2(\phi) \frac{1}{3U}\Big) + K_2 \bigg(\bigg(\frac{U}{3}\log\Big(n + \frac{9}{2} +U\Big)-\left(\frac{n}{3} + \frac{3}{2}\right)\log\Big(1 + \frac{U}{n + 9/2}\Big) \nonumber \\ 
&- \frac{U}{3}\bigg)C_0(\phi)
+\frac{(n + 9/2)\log(n + 9/2 +U)+U\big(\log(n + 9/2 +U)-\log(U)\big)}{9(n + 9/2)U/3}C_2(\phi)\bigg)\Bigg)\cdot2\sqrt{x}.
\end{align}}%
Using the inequality $a\geq\log(1+a)$ and substituting $U = \sqrt{C_2(\phi)/C_0(\phi)}$ gives an upper bound of
\be
\frac{4}{3}\sqrt{x}\sqrt{C_0(\phi)C_2(\phi)}\bigg(K_1+K_2\Big(\log\big(n + 9/2 +\sqrt{C_2(\phi)/C_0(\phi)}\big)+\frac{1}{2}\Big)\bigg) + 4\sqrt{x}\mathfrak{n}_{\sm}(U)C_0(\phi). \nonumber
\ee
For $n \geq 1$, substituting in the values of $K_1$ and $K_2$ and bounding $\log\big(n+9/2+\sqrt{C_2(\phi)/C_0(\phi)}\big)$ with $\log(n) + \big(9/2+\sqrt{C_2(\phi)/C_0(\phi)}\big)/n$ gives the final bound. For $n = 0$, a direct substitution into (\ref{eqn:intEval}) gives the final bound.

\end{proof}

\begin{prop} \label{prop:triv}
Assume $\text{supp}(\phi) \subset [1/2, 5/2]$ and $x \geq 10^6$. Then the sum $\left|\sum_\rho \Phi_x(\rho)\right|$ is bounded above by
\be
.004(n+2) + \Phi(0),
\ee
where the sum ranges over the trivial zeros of $L(s, \sm)$.
\end{prop}

\begin{proof}
A direct calculation gives for $m\geq 0$.
\be
|\Phi_x(-m/2)| \leq \Big(\frac{1}{100,000}\Big)^{m/2}.
\ee
As in Section 7.2 of \cite{RT}, the trivial zeroes occur at most at the negative half-integers with multiplicity at most $1 + \frac{n}{2}$. Additionally, there is a possible zero at $s = 0$ which contributes an additional $\Phi(0)$ term to the sum.  Thus
\be
\sum_{\rho \text{ trivial}}|\Phi_x(\rho)| \leq \frac{n+2}{2}\sum_{m=1}^\infty |\Phi_x(-m/2)| + \Phi(0) \leq .004(n+2)+ \Phi(0). \nonumber
\ee
\end{proof}

\subsection{Error from passing to prime powers}
We bound the error between $\psi_{\sm}(x)$ and $\sum_{p \not| N} U_n(\cos(\theta_p))\chi(p)\phi_x(p)$ to complete the proof of Proposition \ref{prop:3.3analogue}.
\begin{prop} \label{prop:primepwr}
Assume $\text{supp}(\phi) \subset [1/2, 5/2],$ $\mxp \leq 2,$ and $x > 10^6$.  Then
\be
\Big| \sum_{p^j} \Lambda_{\sm}(p^j)\phi_x(p^j) - \sum_{p \not| N}U_n(\cos(\theta_p)) \log(p) \chi(p)\phi_x(p)\Big| \leq (n+1)(3.983\sqrt{x} + 2\log N).
\ee
\end{prop}

\begin{proof}
Using the estimate $|\Lambda_{\sm}(p^j)| \leq (n+1) \log p$, we have

\begin{align}
&\Big| \sum_{p^j} \Lambda_{\sm}(p^j)\phi_x(p^j) - \sum_{p \not| N}U_n(\cos(\theta_p)) \log(p) \chi(p)\phi_x(p)\Big| \nonumber \\ 
&\leq \sum_{p^j, j \geq 2} \left|\Lambda_{\sm}(p^j)\phi_x(p^j)\right| + \sum_{p \mid N}\left|\Lambda_{\sm}(p)\phi_x(p)\right| \nonumber \\
&\leq (n+1)\sum_{p^j, j \geq 2} \log(p) |\phi_x(p^j)| + (n+1)\sum_{p | N} \log(p) |\phi_x(p)|. \nonumber
\end{align}

We recall Rosser and Schoenfeld's \cite{RS} bound of $\psi(x) - \theta(x) < 1.001102 \sqrt{x} + 3x^{1/3}$ for all $x > 0$, and the trivial bound $x^{1/3} < x^{1/2}/10$ for $x > 10^6$. Applying these bounds, the above sum is bounded above for all $x > 10^6$ by
\begin{align}
&(n+1) \mxp \left(\psi(5x/2) - \theta(5x/2)\right) + (n+1) \mxp \log N \nonumber \\
&\leq 2(n+1)\left(1.001102 \sqrt{5x/2} + 3 (5x/2)^{1/3}\right) + 2(n+1) \log N \nonumber \\
&\leq 3.983(n+1)\sqrt{x} + 2(n+1)\log N. \nonumber 
\end{align} 
\end{proof}

\subsection{The Proof of Proposition \ref{prop:3.3analogue}} \label{sec:propproof}
To prove Proposition \ref{prop:3.3analogue}, it simply remains to add the bounds in Propositions \ref{prop:nontriv}, \ref{prop:triv}, and \ref{prop:primepwr}.  Doing so, we obtain for $x\geq 10^6, n \geq 1$ (noting that under our hypotheses, $|\Phi(0)| \leq 8$),
\begin{align}
&\Big| \sum_{p \not| N} U_n(\cos\theta_p)\log(p)\chi(p)\phi_x(p)-\delta_{n,\chi} \Phi(1) x\Big| \nonumber \\
&\leq 2\sqrt{x}\Bigg(\Big(\sqrt{C_0(\phi)C_2(\phi)} + 3C_0(\phi)\Big)\bigg((n+8)\log(n) + (n+1)\Big(\frac{1}{7} + \log(Nq(k-1))\Big)  \nonumber \\
&+ \frac{9}{2}+ \frac{36}{n}\bigg) + \sqrt{C_0(\phi)C_2(\phi)}\bigg(\frac{n}{2} + 7 + \frac{24}{n}\bigg) + C_2(\phi)\bigg(1 + \frac{8}{n}\bigg) + 2(n+1)\Bigg) \nonumber \\
&+ 2(n+1)\log N,
\end{align}
as desired. Similarly in the case where $n = 0$, we have
\begin{align} \label{eqn:n0case}
&\Big|\sum_{p \not| N} U_n(\cos\theta_p)\log(p)\chi(p)\phi_x(p)-\delta_{n,\chi} \Phi(1) x\Big| \nonumber \\
&\leq \sqrt{x}\Big(\sqrt{C_0(\phi)C_2(\phi)}\big(42.96 + 2\log(Nq(k-1))\big) + \big(72.8 + 6\log(Nq(k-1))\big)C_0(\phi) \nonumber \\
&+ 23.56C_2(\phi) + 3.983 \Big) + 2\log N.
\end{align}

\section*{Acknowledgements}
This research was supported by Williams College and the National Science Foundation (grant numbers DMS1659037 and DMS1561945). Casimir Kothari was partially supported by Princeton University. Jesse Thorner was partially supported by a National Science Foundation Postdoctoral Fellowship.    The authors used \emph{Mathematica} for explicit calculations.

\bigskip

\end{document}